\documentclass[reqno,11pt]{amsart}
\usepackage{latexsym,amssymb,amsfonts,amsmath,amsthm}
\usepackage{eucal}
\usepackage{mathtools}
\usepackage[colorlinks,citecolor=blue,urlcolor=blue]{hyperref}
\usepackage{subcaption}
\usepackage{tikz}
\usetikzlibrary{decorations.markings}

\newtheorem{thm}{Theorem}[section]

\newtheorem{lem}[thm]{Lemma}

\theoremstyle{definition}

\theoremstyle{remark}
\newtheorem{rem}[thm]{Remark}
\numberwithin{equation}{section}
\newcommand{\Real}{\mathbb R}
\newcommand{\eps}{\varepsilon}

\newcommand{\F}{\mathcal{F}}

\newcommand{\one}[1]{\mathbf{1}_{\{#1\}}}

\renewcommand{\P}{\mathbb{P}}

\newcommand{\E}{\mathbb{E}}

\DeclareMathOperator{\trace}{tr}

\DeclareMathOperator*{\sign}{sign}

\def\XXint#1#2#3{{\setbox0=\hbox{$#1{#2#3}{\int}$}
     \vcenter{\hbox{$#2#3$}}\kern-.5\wd0}}

\begin{document}

\title[Statistical analysis of the mixed fractional OU process]
{Statistical analysis of the mixed fractional Ornstein--Uhlenbeck process}

\author{P. Chigansky}%
\address{Department of Statistics,
The Hebrew University,
Mount Scopus, Jerusalem 91905,
Israel}
\email{pchiga@mscc.huji.ac.il}

\author{M. Kleptsyna}%
\address{Laboratoire de Statistique et Processus,
Universite du Maine,
France}
\email{marina.kleptsyna@univ-lemans.fr}

\thanks{P. Chigansky is supported by ISF 558/13 grant}
%\subjclass{}%

\date{\today}%
%\dedicatory{}%
%\commby{}%
% ----------------------------------------------------------------
\begin{abstract}

This paper addresses the problem of estimating drift parameter of the Ornstein - Uhlenbeck type process,
driven by the sum of independent standard and fractional Brownian motions. The maximum likelihood estimator 
is shown to be consistent and asymptotically normal in the large-sample limit, using some recent results on 
the canonical representation and spectral structure of mixed processes.

\end{abstract}

\maketitle

%\tableofcontents

\section{Introduction and the main result}

\subsection{Drift estimation problem}
Estimating drift parameter $\theta\in \Real$ from a sample path of the Ornstein--Uhlenbeck type process $X^T=(X_t, t\in [0,T])$:
\begin{equation}
\label{X}
X_t = X_0 + \theta \int_0^t X_s ds + V_t, \quad t\ge 0
\end{equation}
is a prototypical problem in statistical inference of random processes. 
It can be approached in a number of ways, which produce reasonable  
estimators, see e.g. \cite{HN10}, \cite{SV17}. However, without being based on the likelihood function, these 
estimators are asymptotically subefficient as $T\to\infty$, at best up to a finite gap with respect to the information bound. 
Construction and analysis of the likelihood based estimators, on the other hand, requires a convenient formula 
for the likelihood function, which can be hard to find for a given driving process $V$.

In its classic form, with \eqref{X} driven by the standard Brownian motion $B=(B_t, t>0)$, this problem was
extensively studied since 60's. In this case probability measures $\mu^T_\theta$, $\theta\in \Real$ induced by $X^T$ 
are equivalent and the likelihood function is given by the Girsanov exponent:
$$
\frac{d\mu^T_\theta}{d\mu^T_0}(X^T) = \exp \left(\theta \int_0^T X_tdX_t-\frac  1 2 \theta^2 \int_0^T X_t^2 dt \right).
$$
Consequently, the maximum likelihood estimator (m.l.e.) of $\theta$, that is,  the unique maximizer of the likelihood, 
is given by the simple formula
$$
\widehat \theta_T  = \frac{\int_0^T X_t dX_t}{\int_0^T X_t^2 dt}, \quad T>0.
$$
It is asymptotically optimal in the local minimax sense as $T\to\infty$ and 
its limit behavior is determined by the sign of the 
drift parameter $\theta$. In the stable case, for $\theta<0$, the estimation error $\widehat \theta_T-\theta$ 
is asymptotically normal at the usual parametric rate $\sqrt {T}$: 
\begin{equation}\label{LAN}
\sqrt{T} (\widehat \theta_T-\theta)\xrightarrow[T\to\infty]{d}N(0, 2|\theta|), \quad \forall \theta <0,
\end{equation}
where the convergence is in distribution. Entirely different asymptotics emerges in the neutrally stable
and unstable cases, $\theta=0$ and $\theta>0$ respectively.  A comprehensive account of these and other results 
can be found in \cite{Ku04}.

\subsection{Innovation approach}
The likelihood function can be constructed by means of the Girsanov theorem, if the driving 
process $V$ is Gaussian and admits canonical innovation representation (see, e.g., \cite{HH76}), that is,  
if there exists a pair of deterministic kernels $g(s,t)$ and $\widetilde g(s,t)$, such that 
\begin{equation}\label{cir}
V_t = \int_0^t \widetilde g(s,t) dM_s\quad \text{and}\quad  \quad M_t = \int_0^t g(s,t) dV_s, \quad t\ge 0,
\end{equation}
where $M=(M_t, t\ge 0)$ is a continuous martingale with a strictly increasing quadratic variation 
$\langle M\rangle_t$.  The stochastic integrals in \eqref{cir} are defined in some reasonable sense, e.g., 
through  approximation by simple functions. 

Integrating kernel $g(s,t)$ with respect to $X$ gives a semimartingale
\begin{equation}
\label{Zt}
Z_t := \int_0^t g(s,t)dX_s = \theta \int_0^t Q_s(X) d\langle M\rangle_s + M_t
\end{equation}
where 
\begin{equation}\label{Qt}
Q_t(X):=\frac{d}{d\langle M\rangle_t}\int_0^t g(s,t) X_sds.
\end{equation}
and if the filtrations generated by $X$ and $Z$ coincide, then by the Girsanov theorem the measures $\mu^T_\theta$, $\theta \in \Real$ 
are equivalent with the likelihood function of the form 
$$
\frac{d\mu^T_\theta}{d\mu^T_0}(X^T) = \exp \left(\theta \int_0^T Q_t(X)dZ_t-\frac  1 2 \theta^2 \int_0^T Q_t(X)^2 
d\langle M\rangle_t \right).
$$
The m.l.e. is then given by 
\begin{equation}\label{m.l.e.Q}
\widehat \theta_T := \frac{\int_0^T Q_t(X)dZ_t}{\int_0^T Q_t(X)^2 d\langle M\rangle_t}.
\end{equation}

Implementation of this approach however often entails some difficulties.  Firstly, canonical representation \eqref{cir}
may not exist or can be hard to find in a suitable form for a given process $V$. 
Moreover, even if such representation is available, it may not readily reveal 
meaningful information about the estimation error
\begin{equation}
\label{err}
\widehat \theta_T -\theta = \frac{\int_0^T Q_t(X)dM_t}{\int_0^T Q_t(X)^2 d\langle M\rangle_t}.
\end{equation}
Consequently likelihood based estimators have been studied only for a few processes beyond the standard 
Brownian framework.

The innovation approach was realized in \cite{KLeB02} for the Ornstein--Uhlenbeck type process \eqref{X}, 
driven by the fractional Brownian  motion (f.B.m) $B^H=(B^H_t, t\ge 0)$, that is, the centered Gaussian process with covariance function 
\begin{equation}
\label{EBB}
\E B^H_t B^H_s = \frac 1 2\left(t^{2H}+s^{2H}-|t-s|^{2H}\right), \quad s,t\ge 0,
\end{equation}
where $H\in (0,1)$ is the Hurst parameter. For  $H=\frac 1 2$ the f.B.m. coincides with the standard Brownian motion,
but otherwise has different properties, see e.g. \cite{EM02}, \cite{M08}, \cite{BHOZ}.
In particular, for $H>\frac 1 2$ its increments exhibit long-range dependence, which makes f.B.m.  important in 
modeling (see, e.g., \cite{PT17}). The m.l.e. in \cite{KLeB02} is shown to satisfy asymptotics \eqref{LAN}, 
using the canonical representation of the f.B.m., also known in the literature as the Molchan-Golosov transformation 
(see also \cite{NVV99}, \cite{J06}).  

\subsection{The main result}

In this paper we revisit drift estimation problem for the {\em mixed fractional} Ornstein--Uhlenbeck process \eqref{X} driven by
\begin{equation}\label{VBB}
V_t=B_t+B^H_t, \quad t\ge 0
\end{equation}
where $B$ and $B^H$, $H\in (0,1)$ are independent standard and fractional Brownian motions.
Interest in this particular mixture has been triggered by paper \cite{Ch01}, which revealed a number of its curious properties, 
relevant to mathematical finance, see \cite{Ch03}, \cite{BSV07};  some further related results appeared in
\cite{Ch03b}, \cite{BN03}, \cite{vZ07}, \cite{CCK}, \cite{DMS14}.

We will use the canonical representation suggested in \cite{CCK},  based on the martingale 
\begin{equation}
\label{M}M_t := \E(B_t|\F^V_t).
\end{equation}
To this end consider the integro--differential Wiener-Hopf type equation:
\begin{equation}
\label{WHeq}
g(s,t) + \frac {d}{ds} \int_0^t g(r,t) H |s-r|^{2H-1} \sign(s-r)dr = 1, \quad 0<s\ne t \le T.
\end{equation}
By Theorem 5.1 in \cite{CCK} this equation has unique solution for any $H\in (0, 1)$. It is 
continuous on $[0,T]$, and the martingale, defined in \eqref{M}, satisfies 
\begin{equation}\label{martingale}
M_t = \int_0^t g(s,t)dV_t\quad \text{and}\quad \langle M\rangle_t = \int_0^t g(s,t)ds, \quad t\in [0,T],
\end{equation} 
where the stochastic integral is defined for $L^2(0,T)$ deterministic integrands in the usual way (see, e.g., \cite{LS1}). 
By Corollary 2.9 in \cite{CCK}, process $V$ admits canonical representation \eqref{cir}  with 
\begin{equation}
\label{tildeg}
\widetilde g(s,t) := 1- \frac{d}{d\langle M\rangle_s}\int_0^t g(r,s)dr,
\end{equation}
and the m.l.e. of $\theta$ is given by \eqref{m.l.e.Q}.

\medskip
The main result of this paper is the proof of asymptotic normality of the m.l.e.:

\begin{thm}\label{thm}
Let $X=(X_t, t\ge 0)$ be the stable Ornstein--Uhlenbeck process, generated by equation \eqref{X} with drift parameter $\theta<0$ and 
driving process $V=(V_t, t\ge 0)$ defined in \eqref{VBB}. The maximum likelihood estimator 
\eqref{m.l.e.Q} is asymptotically normal with limit \eqref{LAN}, where all moments converge.  

\end{thm}

\section{Proof of Theorem \ref{thm}}

We will first derive the weak limit \eqref{LAN} in Section \ref{sec:2.1} and then prove convergence of moments in Section \ref{sec:2.2}
by the uniform  integrability argument. 

\subsection{Convergence in distribution}\label{sec:2.1}
The proof is inspired by the approach in \cite{KLeB02}.   
In view of \eqref{err}, convergence in distribution \eqref{LAN} follows from (Theorem 1.19 in \cite{Ku04})
$$
\frac 1 T \int_0^T Q_t^2 d\langle M\rangle_t\xrightarrow[T\to\infty]{\P} \frac 1{2 |\theta| }.
$$
We will derive this limit by proving convergence of the Laplace transform
\begin{equation}
\label{enough}
\mathcal{L}_T(\mu):=\E \exp \left(-\mu \frac 1 T \int_0^T Q_t^2 d\langle M\rangle_t\right) \xrightarrow{T\to\infty} 
\exp\left( -\frac\mu{2 |\theta| } \right),
\quad \mu \in \Real.
\end{equation}
It will become clear from the proof, that for any $\mu\in \Real$  the Laplace transform is well defined $\mathcal{L}_T(\mu)<\infty$ 
for all sufficiently large $T$. 

The main difficulty in implementing the approach from \cite{KLeB02} in our setup is the lack of explicit expressions for 
kernels $g(s,t)$ and $\widetilde g(s,t)$.
We will show that the large sample asymptotics of the m.l.e.  is governed in this case 
by a certain singularly perturbed version of integro--differnetial equation \eqref{WHeq}. Asymptotic analysis of this equation 
is carried out in our paper, using approximations of the eigenvalues and eigenfunctions for 
the fractional Brownian noise, obtained recently in \cite{ChK}.

\medskip

The proof is split into several lemmas. The first step is to show that process $Q_t$
admits representation as the stochastic integral with respect to auxiliary observation process $Z_t$ defined \eqref{Zt},  
whose integrand is controlled by derivative of the martingale bracket $d\langle M\rangle_t/dt$. 
This derivative exists and is continuous by Theorem 2.4 in \cite{CCK}.

\begin{lem}\label{lem.l.e.m} Let $Q=(Q_t,t\ge 0)$ be the process defined in \eqref{Qt}, with $X=(X_t, t\ge 0)$ 
being the solution of stochastic equation \eqref{X}, driven by $V=(V_t, t\ge 0)$ from \eqref{VBB}. Then 
$$
Q_t = \int_0^t \psi(s,t) dZ_s,
$$
where
\begin{equation}
\label{psist}
\psi(s,t) = \frac 1 2 \left(\frac{dt}{d\langle M\rangle_t}+\frac{ds}{d\langle M\rangle_s}\right).
\end{equation}
\end{lem}

\begin{proof}
By Corollary 2.9 in \cite{CCK}, $\F^X_t = \F^Z_t$ and
$
\displaystyle X_t = \int_0^t \widetilde g(s,t)dZ_s
$
with
$
\widetilde g(s,t)
$
given in \eqref{tildeg}. 
Consequently,
\begin{align*}
Q_t = &
\frac{d}{d\langle M\rangle_t}\int_0^t g(s,t) X_s ds = \frac{d}{d\langle M\rangle_t}\int_0^t g(s,t)\int_0^s \widetilde g(r,s)dZ_r ds  \\
= & \frac{d}{d\langle M\rangle_t}\int_0^t \left(\int_r^t   g(s,t) \widetilde g(r,s) ds\right) dZ_r \stackrel{\dagger}{=}
\int_0^t \psi(r,t) dZ_r,
\end{align*}
with
$$
\psi(r,t) := \frac{d}{d\langle M\rangle_t} \left(\int_r^t   g(s,t) \widetilde g(r,s) ds\right),
$$
where the equality $\dagger$ holds, since the integrand vanishes at $r=t$.
Note that $\psi(s,t)$ does not depend on $\theta$ and hence we can assume  $\theta=0$ for the rest of the proof. 
Then  
$$
Q_t = \int_0^t \psi(r,t)dM_r
\quad
\text{and}
\quad
\E Q_tM_s = \int_0^s \psi(r,t) d\langle M\rangle_r, \quad s\le t
$$
and consequently 
$$
\psi(s,t) = \frac{\partial}{\partial\langle M\rangle_s}\E  Q_t M_s=
\frac{\partial }{\partial\langle M\rangle_s  }\frac{\partial }{  \partial\langle M\rangle_t} \E  \int_0^t g(r,t) V_r  dr
\int_0^s g(r,s)dV_r.
$$
Set  $v(t):= \frac 1 2\E (B^H_t)^2= \frac 1 2 t^{2H}$, so that the covariance in \eqref{EBB} reads
$$
\E B^H_t B^H_s =  v(t)+  v(s)-  v(|t-s|).
$$
Then we have
\begin{align*}
&
\E  \int_0^t g(\tau,t) V_\tau d\tau \int_0^s g(r,s)dV_r=
\E  \int_0^t \int_0^s g(\tau,t) g(r,s) \frac{\partial}{\partial r} \E V_\tau  V_r dr d\tau =\\
&
\int_0^s \int_0^t g(r,s) g(\tau,t) \frac{\partial}{\partial r}\Big(\tau \wedge r +   v(t)+   v(r)
-  v(|r-\tau|)\Big)  d\tau dr =\\
&
\int_0^s \int_0^t g(r,s) g(\tau,t) \Big(\one{r\le \tau} +   v'(r)- v'(|r-\tau|)\mathrm{sign}(r-\tau)\Big)  d\tau dr =\\
&
\int_0^s g(r,s) \phi(r,t) dr  +\int_0^s g(r,s)v'(r)dr\int_0^t  g(\tau,t) d\tau, 
\end{align*}
where we defined 
$$
\phi(r,t):=  \int_r^t  g(\tau,t)  d\tau -
\int_0^t  g(\tau,t)v'(|r-\tau|)\mathrm{sign}(r-\tau)  d\tau.
$$
Since $g(s,t)$ solves \eqref{WHeq}, we have  
$$
\frac{\partial}{\partial r}\phi(r,t)=
-  g(r,t)
-
\frac{\partial}{\partial r}\int_0^t  g(\tau,t)v'(|r-\tau|)\mathrm{sign}(r-\tau)  d\tau =-1
$$
and, integrating, 
$$
\phi(r,t) = \phi(0,t)-r=
\int_0^t  g(\tau,t)  d\tau +
\int_0^t  g(\tau,t)v'(\tau)   d\tau -r 
=: \langle M \rangle_t+\langle N\rangle_t-r.
$$

Gathering all parts together, we obtain
\begin{equation}\label{plugpsi}
\begin{aligned}
\psi(s,t) = & 
\frac{\partial }{\partial\langle M\rangle_s  }\frac{\partial }{  \partial\langle M\rangle_t} \left(
\int_0^s g(r,s) \big(\langle M \rangle_t+\langle N\rangle_t-r\big) dr  +\langle N\rangle_s \langle M\rangle_t
\right) =\\
&
\frac{\partial }{\partial\langle M\rangle_s  }\frac{\partial }{  \partial\langle M\rangle_t}\big(\langle M \rangle_t+\langle N\rangle_t \big)
\langle M\rangle_s   + \frac{d\langle N\rangle_s }{d \langle M\rangle_s  } =
1 +\frac{d\langle N\rangle_t  }{  d\langle M\rangle_t} +\frac{d\langle N\rangle_s }{d \langle M\rangle_s  }.
\end{aligned}
\end{equation}

Now integrate equation \eqref{WHeq} to get 
$$
\int_0^t g(s,t)ds + \int_0^t \frac {d}{ds} \int_0^t g(r,t) v'(|s-r|)  \sign(s-r)dr ds= t,
$$
or, equivalently,
$
\langle M \rangle_t +2  \langle N \rangle_t =t,
$
where we used the symmetry $g(s,t)=g(t-s,t)$.
Hence 
$$
1 +2  \frac {d\langle N \rangle_t}{d \langle M \rangle_t} =\frac{dt}{d\langle M \rangle_t} 
$$
and plugging this expression into \eqref{plugpsi} gives \eqref{psist}.

\end{proof}

Using the special structure of process $Q_t$, derived in  Lemma \ref{lem.l.e.m},  the
Laplace transform \eqref{enough} can be expressed in terms of solution to an auxiliary 
Riccati differential equation:
\begin{lem}
Consider the Riccati equation  
\begin{equation}\label{Riceq}
\dot \Gamma(t) =  \frac { \theta } 2 A(t) \Gamma(t) +\frac { \theta } 2\Gamma(t)A(t)^\top+B(t) -\frac \mu {2T} \Gamma(t) R(t)\Gamma(t), \quad t\in [0,T],
\end{equation}
subject to the initial condition $\Gamma(0)=0$, where 
\begin{equation}\label{ABR}
\begin{aligned}
&
A(t)=
\begin{pmatrix}
1 & \dfrac 1 {\psi(t,t)} \\
\psi(t,t) & 1
\end{pmatrix}   \quad 
B(t)=
\begin{pmatrix}
\dfrac 1 {\psi(t,t)} &  1\\
1 & \psi(t,t)
\end{pmatrix} \\
&  
R(t) = \begin{pmatrix}
  \psi(t,t)  & 1\\
1 & \dfrac 1{ \psi(t,t) }
\end{pmatrix},
\end{aligned}
\end{equation}
with $\psi(t,t)$ defined in Lemma \ref{lem.l.e.m}.
Let $\mu\in \Real$ be such that \eqref{Riceq} has a continuous solution on $[0,T]$, then 
\begin{equation}
\label{Laplace}
\mathcal{L}_T(\mu) =\exp \left(-  \frac \mu {4T}\int_0^T \trace\big(\Gamma(s) R(s)\big)ds\right)<\infty. 
\end{equation}
\end{lem}

\begin{rem}
The Riccati equation \eqref{Riceq} is well known to have unique continuous solution on {\em any} interval $[0,T]$ for all $\mu\ge 0$
and in this case the expression \eqref{Laplace} follows from the Cameron-Martin formula, see Section 4.1 in \cite{KLeB01}.
We will give a proof, which does not require $\mu$ to be positive. This is essential for convergence of moments in \eqref{LAN}, 
as explained in Section \ref{sec:2.2} below. 
In general Riccati equations with positive definite quadratic term, corresponding to $\mu<0$ in our case, can be guaranteed 
to have only local solution, which can explode in finite time. Global solvability of \eqref{Riceq} for any $\mu\in \Real$ 
is proved in Lemma \ref{lem-cond} below. 
\end{rem}

\begin{proof}
By Lemma \ref{lem.l.e.m}
$$
Q_t = \int_0^t \psi(r,t) dZ_r = \frac1 2  \psi(t,t)Z_t +
\frac 1 2 \int_0^t  \psi(r,r)  dZ_r.
$$
Let $\displaystyle Y_t = \int_0^t  \psi(r,r)  dZ_r$, then
\begin{align*}
dZ_t =\,
&
\theta   Q_t  d\langle M\rangle_t + dM_t =
\frac \theta 2   \psi(t,t)Z_t d\langle M\rangle_t +
\frac \theta 2     Y_t d\langle M\rangle_t + dM_t =\\
&
\frac  \theta  2   Z_t dt +
\frac  \theta  2  Y_t \frac 1{\psi(t,t)}dt + \frac 1{\sqrt{\psi(t,t)}}dW_t,
\end{align*}
where we used \eqref{psist} and defined the standard Brownian motion $W_t=\int_0^t \sqrt{\psi(s,s)}dM_s$.
Similarly,
\begin{align*}
dY_t =
&
  \frac \theta 2  \psi(t,t)^2Z_td\langle M\rangle_t    +\frac \theta 2   \psi(t,t) Y_td\langle M\rangle_t
 + \psi(t,t) dM_t =\\
&
\frac \theta 2 \psi (t,t)Z_t dt    +  \frac \theta 2 Y_tdt
 + \sqrt{\psi(t,t)} dW_t.
\end{align*}
Hence the vector $\xi_t = (Z_t, Y_t)^\top$ solves the linear system of It\^o stochastic differential equations
\begin{equation}\label{zetaeq}
d\xi_t = \frac {\theta} 2 A(t)\xi_t dt + b(t) dW_t
\end{equation}
with $A(t)$ defined in \eqref{ABR} and 
$
b(t)^\top =
\begin{pmatrix}
 \dfrac 1 {\sqrt{ \psi(t,t)}} ,
 \sqrt{\psi(t,t)}
\end{pmatrix}.
$
The Laplace transform \eqref{enough} satisfies
\begin{align*}
&
\mathcal{L}_T(\mu)=\E  \exp\left(-   \frac \mu T \int_0^T Q^2_{t} d\langle M \rangle_t\right)=
\E  \exp\left(-   \frac \mu {4T}\int_0^T \big(\psi(t,t)Z_t+Y_t\big)^2 d\langle M \rangle_t\right)  =\\
&
\E  \exp\left(-\frac \mu {4T}\int_0^T \Big(\sqrt{\psi(t,t)}Z_t+\frac 1{\sqrt{\psi(t,t)}}Y_t\Big)^2  dt\right)  = 
\E  \exp\left(-\frac {\tilde\mu} {2}\int_0^T  \big(q(t)^\top\xi_t\big)^2      dt\right),
\end{align*}
where we defined $\tilde \mu :=\mu/(2T)$ and 
$
q(t)^\top := \begin{pmatrix}
\sqrt{\psi(t,t)} ,
\dfrac 1 {\sqrt{ \psi(t,t)}}
\end{pmatrix}.
$

The process $\big(q(t)^\top\xi_t, t\in [0,T]\big)$ is Gaussian with zero mean and 
continuous covariance function $K(t,s)=q(t)^\top \E\xi_t\xi_s^\top q(s)$. 
The eigenvalues $\lambda_1(T)\ge \lambda_2(T)\ge ...$ of the corresponding covariance operator are nonnegative and converge to zero and
\begin{equation}\label{Ltildemu}
\mathcal{L}_T(\mu) = \prod_{j=1}^\infty \frac1 {\sqrt{1+\tilde \mu\lambda_j(T)}} = \frac 1{ \sqrt{D(\tilde \mu)}},\quad \tilde \mu>- 1/\lambda_1(T)
\end{equation}
where $D(\tilde \mu)$ is the Fredholm determinant of $K$ (here and below we use the same notation for integral operators and their kernels). 
By Proposition  IV.7.2$^\circ$ in \cite{GohbergKrein70} %, page 185, 
operator $K$ admits the factorization
\begin{equation}\label{KVV}
(I+\tilde\mu K) = (I+V_+)(I+V_-) 
\end{equation}
where $V_+$ and $V_-$ are left and right Volterra operators. Their kernels $V_+(t,s)$ and $V_-(t,s)$ vanish
for $s<t$ and $t<s$ respectively, are continuous on the complementary triangles and coincide on the diagonal. 
Hence the operator $V_++V_-$ has continuous kernel 
$$
V(t,s) = \begin{cases}
V_+(t,s) & t\ge s\\
V_-(t,s) & t<s
\end{cases}
$$
and is therefore trace class. By an identity due to Krein (see Theorem on page 232 in \cite{Bal73}): 
\begin{equation}\label{Dtildemu}
\log D(\tilde \mu) = \trace (V_++V_-) = \int_0^T V(s,s)ds.
\end{equation}

Since kernel $K(s,t)$ is symmetric around diagonal, so is $V(s,t)$ and after a change of variables, factorization \eqref{KVV} reduces to 
the Riccati-Volterra equation (see (7.5) in \cite{GohbergKrein70}):
\begin{equation}\label{RV}
\tilde \mu K(t,s) =  V(t,s) + \int_0^s V(t, r)V(s,r)dr, \quad s<t.
\end{equation}
In our case it can be solved using differential Riccati equation \eqref{Riceq} as follows. 
Since $\xi_t$ solves linear stochastic equation \eqref{zetaeq},
\begin{equation}\label{qFPq}
K(t,s) = q(t)^\top F(t,s)P(s)q(s) 
\end{equation}
where  $F(t,s)$ is the fundamental solution of the equation $\dot x_t  = \frac \theta 2 A(t) x_t$ and $P(s)$ solves the Lyapunov 
differential equation 
\begin{equation}
\label{Lyapeq}
\dot P(t)  = \frac \theta 2 A(t) P(t) + \frac \theta 2 P(t)^\top A(t)^\top + b(t)b(t)^\top, \quad t>0
\end{equation}
subject to $P(0)=0$. 
In view of \eqref{qFPq}, it makes sense to look for solutions of the Riccati-Volterra equation \eqref{RV} in the form
\begin{equation}\label{Vts}
V(t,s)=\tilde \mu\, q(t)^\top F(t,s)\Gamma(s)q(s).
\end{equation}
Let us show that this function indeed solves \eqref{RV}, if $\Gamma(t)$ is a continuous solution of \eqref{Riceq}.
To this end, we have 
$$ 
\tilde \mu K(t,s) - V(t,s) - \int_0^s V(t, r)V(s,r)dr =
\tilde \mu q(t)^\top F(t,s)\Delta(s)
q(s) 
$$
where 
$$
\Delta(s):=
P(s)-  \Gamma(s)  -   \tilde \mu\int_0^s   F(s,r)\Gamma(r)q(r) 
q(r)^\top \Gamma(r)^\top F(s,r)^\top  
dr. 
$$
In view of  \eqref{Lyapeq} and \eqref{Riceq} this function satisfies the linear equation
$$
\dot \Delta(s) = 
%&
%\dot P(s)-  \dot \Gamma(s)  -\tilde \mu \Gamma(s)q(s) q(s)^\top \Gamma(s)^\top
%-\frac \theta 2 A(s) \big(P(s)-\Gamma(s)-\Delta(s)\big)
%-\big(P(s)-\Gamma(s)-\Delta(s)\big) \frac \theta 2 A(s)^\top =\\
%&     
\frac \theta 2 A(s)\Delta(s)    +\Delta(s)\frac \theta 2 A(s)^\top, \quad s\ge 0
$$
subject to $\Delta(0)=0$, which implies $\Delta(s)\equiv 0$ for all $s\ge 0$ by uniqueness of the solution. 
Hence $V(t,s)$ in \eqref{Vts} solves \eqref{RV} and  
plugging it into \eqref{Dtildemu} and \eqref{Ltildemu} and setting $\tilde \mu = \mu/(2T)$ we obtain formula \eqref{Laplace}. 
   
\end{proof}
The next lemma establishes solvability of the Riccati equation \eqref{Riceq} and formulates sufficient conditions for asymptotic 
normality of the m.l.e. in terms of the innovating martingale bracket:

\begin{lem}\label{lem-cond}
Assume that the quadratic variation of the martingale in \eqref{martingale} satisfies the growth conditions 
\begin{equation}\label{ass}
\int_0^\infty \left(\frac d {dt} \log \frac{d}{dt}\langle M \rangle_t \right)^2dt <\infty
\end{equation}
and
\begin{equation}
\label{check2}
\frac 1 t \max \left( \frac{dt}{d\langle M\rangle_t},\frac{d\langle M\rangle_t}{dt}\right)\xrightarrow{t\to\infty}0.
\end{equation}
Then the following assertions hold:

\medskip 

\begin{enumerate}
\addtolength{\itemsep}{0.7\baselineskip}
\renewcommand{\theenumi}{\alph{enumi}}

\item for any $\mu\in \Real$ the Riccati equation \eqref{Riceq} has unique continuous solution on $[0,T]$ for all sufficiently large 
$T$

\item the Laplace transform converges to the limit \eqref{enough} and

\item\label{lem-cond:c} satisfies the bound 
\begin{equation}\label{muexpbnd}
\mathcal{L}_T(\mu)\le c_1 \mu^2 \exp \Big(-c_2\sqrt{\mu} \Big),\quad \forall \mu> 0, 
\end{equation} 
where $c_1$ and $c_2$ are positive constants, independent of $T$.
\end{enumerate}

\end{lem}

\begin{proof}\

\medskip
\noindent 
(a) 
Fix any $\mu\in \Real$ and let  $\Phi_1(t)$ and $\Phi_2(t)$ be the solutions of the linear system: 
\begin{equation}
\label{PhiPhi}
\begin{aligned}
& \dot\Phi_1(t)  = -\frac { \theta } 2 \Phi_1(t) A(t) + \frac{\mu}{2T}\Phi_2(t) R(t)\\
& \dot\Phi_2(t) = \phantom{+} \Phi_1(t) B(t) +\frac { \theta } 2 \Phi_2(t) A(t)^\top
\end{aligned}
\end{equation}
subject to $\Phi_1(0)=I$ and $\Phi_2(0)=0$. By continuity $\Phi_1(t)$ remains nonsingular on a vicinity of the origin and  
the direct calculation shows that $\Gamma(t)=\Phi^{-1}_1(t)\Phi_2(t)$ solves \eqref{Riceq}. 
We will argue that $\Phi_1(t)$ in fact remains nonsingular on the interval $[0,T]$, provided $T$ is chosen large enough 
and thus $\Gamma(t)$ is a global solution for all such $T$.  

To this end let
$
J = \left(\begin{smallmatrix}
0 & 1 \\
1 & 0
\end{smallmatrix}
\right)
$
and note that $R(t) = J A(t)$, $B(t) = A(t)J$ and $J A(t) J=A(t)^\top$.
If we now define  $\widetilde \Phi_2(t):= \Phi_2(t)J$ and multiply the second equation in \eqref{PhiPhi} by $J$ from the right,  we obtain the system
\begin{equation}
\label{sys12}
\begin{aligned}
& \dot\Phi_1(t)  = -\frac { \theta } 2 \Phi_1(t) A(t) + \frac{\mu}{2T}\widetilde \Phi_2(t)  A(t)\\
& \dot{\widetilde \Phi}_2(t) = \phantom{+} \Phi_1(t) A(t) +\frac { \theta } 2 \widetilde \Phi_2(t)   A(t) 
\end{aligned}
\end{equation}
subject to $\Phi_1(0)=I$ and $\Phi_2(0)=0$.
Let $T$ be large enough so that $\left(\frac {\theta }2\right)^2+\frac {\mu } {2T}>0$,
then matrix
$
\left(
\begin{smallmatrix}
-\frac { \theta } 2 & \frac \mu {2T} \\
1 &  \frac { \theta } 2
\end{smallmatrix} 
\right)
$
has two real eigenvalues $\pm \gamma_T$ with $\gamma_T = \sqrt{\left(\frac {\theta }2\right)^2+\frac {\mu } {2T}}$ and the corresponding eigenvectors
$$
v_+ = \begin{pmatrix}
a^+_T \\
1
\end{pmatrix} \quad\text{and} \quad 
v_- = \begin{pmatrix}
a^-_T \\
1
\end{pmatrix},
$$
where $a^{\pm}_T=-\frac {\theta} 2 \pm \gamma_T$.
Diagonalizing \eqref{sys12} we obtain 
$$
\Phi_1(t) = a^+_T \Upsilon_1(t) + a^-_T\Upsilon_2(t),
$$
where $\Upsilon_1(t)$ and $\Upsilon_2(t)$ solve decoupled equations 
\begin{equation}
\label{Upsilon}
\begin{aligned}
&\dot \Upsilon_1(t) =\;\;\, \gamma_T\Upsilon_1(t) A(t) \\
&\dot \Upsilon_2(t) = -\gamma_T \Upsilon_2(t) A(t)
\end{aligned}
\end{equation}
subject to $\Upsilon_1(0) =  -\Upsilon_2(0)=  I/(2\gamma_T)$.
Hence 
\begin{equation}
\label{logdet}
\begin{aligned}
\log \det \big(\Phi_1(t)\big)  = &  \log \det \big(a^+_T \Upsilon_1(t) + a^-_T\Upsilon_2(t)\big) =  \\
&
\log \det \big(a^+_T \Upsilon_1(t)\big) +\log \det\Big(I + \frac{a^-_T}{a^+_T}\Upsilon_1^{-1}(t)\Upsilon_2(t)\Big),
\end{aligned}
\end{equation}
where inverse $\Upsilon_1^{-1}(t)$ exists at least on some vicinity of the origin by continuity of the solution. 

Let us show that $\log \det \big(\Phi_1(t)\big)$ remains finite on $[0,T]$. To this end, the first term on the right in \eqref{logdet} 
satisfies:
\begin{align}
&
\label{term1}
\log \det \big(a^+_T \Upsilon_1(t)\big)   = 
\log  (a^+_T)^2 +  \log \det\Upsilon_1(0) + \gamma_T \int_0^t  \trace A(s) ds  = \\
&
\nonumber
\log  \Big(\frac {|\theta|} 2 + \gamma_T\Big)^2 -  \log (2\gamma_T)^2 + 2t\gamma_T   =
-\log 4 + 2\log  \Big(1+\frac {|\theta|} 2 \frac 1 {\gamma_T}\Big)    + 2t\gamma_T,
\end{align}
where we used equality $\trace A(t)=2$. Since the last two terms are positive this implies 
$$
\log \det \big(a^+_T \Upsilon_1(t)\big) \ge -\log 4, \quad t\in [0,T].
$$
To bound the second term in \eqref{logdet}, note that 
\begin{equation}\label{bndaa}
\left|{a^-_T}/{a^+_T}\right| = 
%\frac{
%\frac {|\theta|} 2 - \sqrt{\left(\frac {\theta }2\right)^2+\frac {\mu } {2T}}
%}
%{
%\frac {|\theta|} 2 + \sqrt{\left(\frac {\theta }2\right)^2+\frac {\mu } {2T}}
%}=
\frac{   \frac {|\mu| } {2T} 
}
{
\left(\frac {|\theta|} 2 + \sqrt{\left(\frac {\theta }2\right)^2+\frac {\mu } {2T}}\right)^2
}\le \frac{|\mu|}{2\theta^2}\frac 1 T
\end{equation}
and hence it will suffice to show that 
\begin{equation}\label{Upsbnd}
\frac 1 T \sup_{t\le T}\|\Upsilon_1^{-1}(t)\Upsilon_2(t)\|\xrightarrow{T\to\infty}0,
\end{equation}
where $\|\cdot\|$ denotes the matrix norm, induced by Euclidean norm on $\Real^2$.
To prove this limit we will need an estimate for the solution of the second equation in \eqref{Upsilon}.
Define  
$$
g_t :=\frac 1{\sqrt{\psi(t,t)}} = \sqrt{\frac{d\langle M\rangle_t}{dt}}
$$
and  fix an arbitrary vector $v\in \Real^2$, then
$$
v_t  :=  \begin{pmatrix}
g_t & 0 \\
0 & 1/ {g_t}
\end{pmatrix} 
\Upsilon^\top_2(t) v
$$
solves the equation $\dot v_t =  H(t) v_t$  with symmetric matrix
$$
H(t) = \begin{pmatrix}
-\gamma_T +\dot g_t/g_t& -\gamma_T \\
-\gamma_T & -\gamma_T   -\dot g_t/g_t
\end{pmatrix}.
$$
%\begin{align*}
%\dot v_t =
%&
%v^\top \dot \Upsilon_2(t)\begin{pmatrix}
%g_t & 0 \\
%0 & 1/ {g_t}
%\end{pmatrix}
%+v^\top  \Upsilon_2(t)
%\begin{pmatrix}
%\dot g_t & 0 \\
%0 & -\dot g_t/ {g^2_t}
%\end{pmatrix}
% =\\
%&
%-\gamma_T v_t
%\begin{pmatrix}
%1/g_t & 0 \\
%0 &  g_t
%\end{pmatrix}
%A(t)
%\begin{pmatrix}
%g_t & 0 \\
%0 & 1/ {g_t}
%\end{pmatrix}
%+v_t
%\begin{pmatrix}
%1/g_t & 0 \\
%0 &  g_t
%\end{pmatrix}
%\begin{pmatrix}
%\dot g_t & 0 \\
%0 & -\dot g_t/ {g^2_t}
%\end{pmatrix} =\\
%&
% v_t
%\begin{pmatrix}
%-\gamma_T +\dot g_t/g_t& -\gamma_T \\
%-\gamma_T & -\gamma_T   -\dot g_t/g_t
%\end{pmatrix}.
%\end{align*}
The maximal eigenvalue of this matrix is 
$$
-\gamma_T + \sqrt{\gamma_T^2 + \left(\dot g_t/g_t\right)^2} \le   \frac 1{2\gamma_T}\left(\dot g_t/g_t\right)^2
$$
and thus, under assumption \eqref{ass} and since $g_0=1$ (see Theorem 2.4 (ii) in \cite{CCK})
$$
\|v_t\|_2 \le \frac {\|v\|_2 } {2\gamma_T} \exp \left(\frac 1{2\gamma_T}\int_0^t \left(\dot g_s/g_s\right)^2ds\right) \le  C \|v\|_2, \quad t>0
$$
with the same constant $C$ for all $T$ large enough.  Hence 
$$
\|\Upsilon_2(t)\|  \le C \max (g_t,1/g_t).
$$

Further, note that
$$
\frac{d}{dt} \Upsilon_1^{-1}(t) =   -\gamma_T A(t) \Upsilon_1^{-1}(t)
$$
which under transposition  and multiplication by $J$ from the right  becomes
$$
\frac{d}{dt}  \big(\Upsilon_1^{-\top}(t)J\big)  = 
%-\gamma_T  \Upsilon_1^{-\top}(t) A^\top(t) J=-\gamma_T  \Upsilon_1^{-\top}(t) JJ A^\top(t) J  = 
-\gamma_T  \big(\Upsilon_1^{-\top}(t) J\big)A(t),
$$
that is, $\Upsilon_1^{-\top}(t)J$ and $\Upsilon_2(t)$ solve the same equation.
Therefore
\begin{equation}
\label{Ups}
\|\Upsilon_1^{-1}(t)\Upsilon_2(t)\|\le \|\Upsilon_1^{-1}(t)J\| \|\Upsilon_2(t)\| \le C^2 \max \big(g^2_t, 1/g^2_t\big),
\end{equation}
and \eqref{Upsbnd} holds by continuity of $g_t$ and assumption \eqref{check2}. This shows that for any fixed $\mu\in \Real$, function 
$\Gamma(t)=\Phi^{-1}_1(t)\Phi_2(t)$ solves \eqref{Riceq} on $[0,T]$ for all sufficiently large $T$. 

\medskip
\noindent 
(b) For a fixed $\mu\in \Real$, let $T$ be large enough, so that Riccati equation \eqref{Riceq} has unique solution on $[0,T]$ 
and the Laplace transform satisfies \eqref{Laplace}. Multiplying the first equation in \eqref{PhiPhi} by $\Phi^{-1}(t)$ gives
$$
\Phi^{-1}_1(t)\dot \Phi_1(t) = -\frac {\theta} 2   A(t) + \frac{\mu}{2T}\Gamma(t) R(t),
$$
and since $\trace A(t) =2$ 
$$
\frac \mu {2T} \trace\big(\Gamma(t) R(t)\big) = \trace\big(\Phi^{-1}_1(t)\dot \Phi_1(t)\big) +   \theta.
$$
By the Liouville formula $\trace\big(\Phi_1^{-1}(t)\dot \Phi_1(t)\big)=\dfrac d {dt} \log \det \big(\Phi_1(t)\big)$ and hence 
\begin{equation}
\label{Lfla}
\frac \mu {2T} \int_0^T \trace\big(\Gamma(t) R(t)\big) dt =   \log \det \big(\Phi_1(T)\big)  +   \theta T.
\end{equation}
Since 
$$
\gamma_T=\sqrt{\Big(\frac\theta 2\Big)^2+\frac \mu {2T}}= \frac{|\theta|}{2}+ \frac {\mu}{2|\theta|}\frac 1{T} + O(T^{-2}),
$$
by \eqref{term1} we have  
\begin{align*}
&
\log \det \big(a^+_T \Upsilon_1(T)\big) + \theta T  = 
-\log 4 + 2\log  \Big(1+\frac {|\theta|} 2 \frac 1 {\gamma_T}\Big)    + 2T\gamma_T  - |\theta| T
\xrightarrow[T\to\infty]{}\frac {\mu }{|\theta|}.
\end{align*}
The claimed limit \eqref{enough} is obtained by plugging this, \eqref{bndaa} and \eqref{Upsbnd} and \eqref{logdet} into \eqref{Lfla} and \eqref{Laplace}. 

\medskip 
\noindent 
(c) Riccati equation \eqref{Riceq} is well known to have unique continuous solution for any $\mu>0$ on any interval $[0,T]$.
Hence by \eqref{term1}
\begin{align*}
&
\log \det \big(a^+_T \Upsilon_1(T)\big) +\theta T  = \\
&
-\log 4 + 2\log  \Big(1+\frac {|\theta|} 2 \frac 1 {\gamma_T}\Big)      +  T\left(\sqrt{\theta^2+\frac{2\mu }{T}}  -\sqrt{\theta^2}\right) \ge \\
&
-  \log 4 + \frac T 2\int_{0}^{\frac{2\mu }{T}} \frac 1{\sqrt{\theta^2+x}}dx \ge 
-  \log 4 +      \frac {\mu }{2\sqrt{(\frac\theta 2)^2+\frac{\mu }{2T}}}\ge 
-  \log 4 +      \frac \mu {\sqrt{\theta^2+\mu}}
\end{align*}
where the last bound holds for all $T\ge 2$. The convergence in \eqref{Upsbnd} is uniform over $\mu \ge 0$, since constant $C$ in \eqref{Ups}
can be chosen independently of $\mu$ in this case.  
Hence in view of \eqref{bndaa}, the second term in \eqref{logdet} is bounded by $\log (c\mu^2)$ with a constant $c$, independent of $T$ and $\mu$. The bound \eqref{muexpbnd} now follows from the formulas  \eqref{Lfla} and \eqref{Laplace}. 

\end{proof}

It is left to check the conditions of Lemma \ref{lem-cond}, which we do separately for $H$ less and greater than $1/2$.
Below the brief notation $f_T \sim g_T$ is used, whenever $f_T = C g_T (1+o(1))$ as $T\to\infty$ with a nonzero constant $C$. 

\begin{lem}
For $H>\frac 1 2$
$$
\frac d{dT} \langle M\rangle_T  \sim T^{1-2H} \quad \text{and}\quad
\left(\frac d {dT}\log  \frac d{dT} \langle M\rangle_T\right)^2 \sim T^{-2},
\quad \text{as\ \ } T\to\infty
$$
and thus the conditions of Lemma \ref{lem-cond} hold.
\end{lem}

\begin{proof}
For  $H>\frac 1 2$ the derivative and integration in \eqref{WHeq} can be interchanged and it takes the form of 
integral equation 
$$
g(s,t) +   \int_0^t g(r,t) c_H |s-r|^{2H-2}dr = 1, \quad 0<s<t \le T,
$$ 
where $c_H=H(2H-1)$. 
By Theorem 2.4, \cite{CCK} in this case 
$$
\langle M\rangle_T = \int_0^T g^2(t,t)dt,
$$
with $g(t,t)> 0$ for all $t\ge 0$. 

Define small parameter $\eps := T^{1-2H}$, then the function
$u_\eps(x):=T^{2H-1}g(xT,T)$ solves the integral equation 
\begin{equation}\label{WHeqmf.B.m.}
\eps u_\eps(x) + \int_0^1 c_H|y-x|^{2H-2}u_\eps(y) dy =1, \quad x\in [0,1]
\end{equation}
and, moreover,
\begin{equation}\label{MTeps}
\frac {d \langle M \rangle_T}{dT}     = g^2(T,T)= \eps^2 u^2_\eps (1).
\end{equation}
For any $\eps>0$ the equation of the second kind \eqref{WHeqmf.B.m.} has a unique solution, continuous on the closed 
interval $[0,1]$ (see e.g. \cite{VP80}). For $\eps=0$ it degenerates to the equation of the first kind, whose unique 
solution is known in a closed form \cite{LB98}:
$$
u_0(x) = a_H x^{\frac 1 2-H}(1-x)^{\frac 1 2 -H}, \quad x\in (0,1)
$$
where $a_H$ is an explicit constant. Note that $u_0(x)$ explodes at the endpoints of the interval and therefore 
it is reasonable to expect that $u_\eps(1)\to\infty$ as $\eps\to 0$. 

To estimate the growth of $u_\eps(1)$ we will use the following asymptotic approximations for the ordered sequence 
of eigenvalues  
and scalar products with the corresponding eigenfunctions for the integral operator in \eqref{WHeqmf.B.m.}
(see Theorem 2.3 and Lemma 6.9 in \cite{ChK}): 
\begin{equation}\label{Hge}
\begin{aligned}
\lambda_n & = c_1  n^{1-2H} (1+o(1))\\
\langle (\cdot)^{-\beta} , \varphi_n\rangle & = c_2   n^{\beta-1} (1+o(1)) \\
\langle 1, \varphi_{2n-1}\rangle &   = c_3 n^{-\frac 1 2-H}  (1+o(1))
\end{aligned}
\qquad n\to\infty
\end{equation}
where $c_j$'s are positive constants and $\beta\in (0,1)$. The eigenfunctions with even indices are antisymmetric 
around the midpoint of the interval and hence $\langle 1, \varphi_{2n}\rangle=0$. 

Taking  scalar product of both sides of \eqref{WHeqmf.B.m.} gives
\begin{equation}
\label{equeps}
\begin{aligned}
 u_\eps(1) =&
\;
\eps^{-1} \int_0^1 \big(u_0(x)-u_\eps(x)\big) c_H (1-x)^{2H-2}dx = \\
&
\sum_{n\; \text{odd}}\langle 1, \varphi_n\rangle
 \frac{1}{\lambda_n(\eps+\lambda_n)}
 \int_0^1 \varphi_n(x) c_H (1-x)^{2H-2}dx.  
\end{aligned}
\end{equation}
The estimates in \eqref{Hge} imply that this series converges for all $\eps>0$ and diverges to $+\infty$ as $\eps\to 0$.
Contribution of any finite number of summands is bounded as $\eps\to 0$ and therefore can be neglected. Consequently,   
the limiting behavior of the series in \eqref{equeps} does not change, if all the sequences are replaced by their leading 
asymptotic terms from \eqref{Hge}:
\begin{equation}
\label{ueps1}
u_\eps(1) =  C \sum_{n=1}^\infty 
 \frac{ n^{-\frac 1 2-H} }{ \eps+n^{1-2H} } \big(1+o(1)\big)\quad \text{as}\ \eps\to 0.
\end{equation}
where $C>0$ absorbs all the constants.    
This series can be estimated by an integral:
\begin{align*}
&
\sum_{n=1}^\infty\frac{n^{-\frac 1 2-H}}{ \eps+  n^{1-2H} } \le 1 +
\int_{2}^\infty \frac{(x-1)^{-\frac 1 2-H}}{ \eps+  x^{1-2H} }dx= \\
&
1 +
\eps^{-\frac 1  2}\int_{2\eps^{\frac 1{2H-1}}}^\infty \frac{  y^{2H-1}(y-\eps^{ \frac 1{2H-1}})^{-\frac 1 2-H}}{   y^{2H-1} +1}dy
= 
\eps^{-\frac 1  2}\int_{0}^\infty \frac{  y^{H-\frac 3 2}}{   y^{2H-1} +1}dy (1+o(1)).
\end{align*}
Analogous calculation yields the same lower bound and in view of \eqref{MTeps} and \eqref{ueps1}
we obtain the claimed asymptotics: 
\begin{equation}\label{ueps}
\frac {d \langle M \rangle_T}{dT}    = \eps^2 u^2_\eps (1) \sim \eps = T^{1-2H}.
\end{equation}

The second condition is verified similarly: 
\begin{align*}
&
\frac{d}{d\eps} u_\eps (1) =
\frac{d}{d\eps}\sum_{n\; \text{odd}}\frac{\langle \varphi_n, 1\rangle}{\lambda_n(\eps+\lambda_n)}\int_0^1 \varphi_n(x) c_H |x-1|^{2H-2}dx \sim \\
&
- \sum_{n\; \text{odd}}\frac{n^{-\frac 1 2-H}}{ (\eps+  n^{1-2H})^2}
\sim
-    \int_1^\infty \frac{x^{-\frac 1 2-H}}{ (\eps+  x^{1-2H})^2}dx\sim
-  \eps^{-\frac 3 2}  \int_0^\infty \frac{y^{3H-\frac 5 2}}{ (  y^{2H-1}+1  )^2}dy,
\end{align*}
and hence
\begin{multline*}
\left(\frac d {dT}\log  \frac d{dT} \langle M\rangle_T\right)^2 = \left(\frac d {dT}\log g(T,T) \right)^2= \left(\frac d{dT} \log T^{1-2H} u_\eps(1)\right)^2\le \\
\frac 2 {T^2} +2\left(\frac {d\eps}{dT}\frac{d}{d\eps} \log  u_\eps(1)\right)^2 =
 \frac 2 {T^2} +2\left(\frac {d T^{1-2H}}{dT} \frac{  \frac d{d\eps}u_\eps(1) }{u_\eps(1)}\right)^2
\sim \frac 1{T^2}.
\end{multline*}

\end{proof}

\begin{lem}
For $H<\frac12$,
$$
\frac d{dT} \langle M\rangle_T  \sim \; \text{const.} \quad \text{and}\quad
\left(\frac d {dT}\log  \frac d{dT} \langle M\rangle_T\right)^2 \sim T^{-2}
\quad \text{as\ \ } T\to\infty
$$
and thus the conditions of Lemma \ref{lem-cond} hold.
\end{lem}

\begin{proof}
For   $H<\frac  12$ the equation \eqref{WHeq} takes the form (see Theorem 5.1 in \cite{CCK}):
\begin{equation}
\label{fraceq}
c_H (\Phi g)(s) + \frac{2-2H}{\lambda_{H}}(\Psi g)(s,t)s^{1-2H}=c_H(\Phi 1)(s), \quad s\in (0,t],
\end{equation}
where 
\begin{align*}
(\Psi g)(s,t) &= -2H\frac{d}{ds}\int_s^t g(r,t)
r^{H-\frac 12}(r-s)^{H-\frac 12}\,dr \\
(\Phi f)(s)  &= \frac{d}{ds}\int_0^s f(r)r^{\frac 12-H} (s-r)^{\frac 12-H}\,dr.
\end{align*}
Moreover, by Theorem 2.4 in \cite{CCK},
$$
\sqrt{\frac d{dt} \langle M\rangle_t} = \sqrt{\frac{2-2H}{\lambda_H}} t^{\frac 1 2-H}(\Psi g)(t,t)  =: p(t,t),
$$
and it follows from  \eqref{fraceq} that
\begin{align*}
&
p(t,t) =  \sqrt{\frac{\lambda_H}{2-2H}}t^{H-\frac 1 2}c_H \Big((\Phi 1)(t)-(\Phi g)(t)\Big) =\\
&
c_H\sqrt{\frac{\lambda_H}{2-2H}} (\tfrac 1 2-H )
t^{H-\frac 1 2} \int_0^t \big(1-g(r,t)\big)r^{\frac 1 2-H}(t-r)^{-\frac 1 2-H}dr.
\end{align*}
Let $\eps:=T^{2H-1}$ and define $u_\eps(u):= g(uT,T)$, $u\in [0,1]$, then
$$
p(T,T) = C \eps^{-\frac 1 2}
\int_0^1 \big(1-u_\eps(x)\big)u^{\frac 1 2-H}(1-x)^{-\frac 1 2-H}dx,
$$
with a constant $C>0$. The function  $u_\eps$ solves the equation
\begin{equation}\label{WHeqHless}
\eps   u_\eps + K_H^{-1} u_\eps = K_H^{-1} 1,
\end{equation}
where $K_H$ stands for the operator in \eqref{WHeq}. For $H<\frac 1 2$ the inverse $K_H^{-1}$ turns out to be an integral operator 
with a certain  weakly singular kernel (see (iv) of Theorem 5.1 in \cite{CCK}). 

The limit equation is uniquely solved by $u_0\equiv 1$ and hence
$$
p(T,T) = C  \eps^{-\frac 1 2}
\int_0^1 \big(u_0(x)-u_\eps(x)\big)x^{\frac 1 2-H}(1-x)^{-\frac 1 2-H}dx.
$$
Since $u_0, u_\eps\in L^2(0,1)$,
\begin{align*}
u_0 -u_\eps = & \sum_n \langle 1,\varphi_n\rangle \varphi_n -
\sum_n \frac{\langle K_H^{-1} 1,\varphi_n\rangle}{\eps+\lambda_n^{-1}}\varphi_n \\
=&  \sum_n \langle 1,\varphi_n\rangle \varphi_n -
\sum_n \frac{\lambda_n^{-1}\langle  1,\varphi_n\rangle}{\eps+\lambda_n^{-1}}\varphi_n=
\sum_n \frac{\eps}{\eps + \lambda_n^{-1}}\langle 1,\varphi_n \rangle\varphi_n.
\end{align*}
Define $h(u):=C u^{\frac 1 2-H}(1-u)^{-\frac 1 2-H}$, then
\begin{equation}
\label{pTT}
p(T,T)=   \eps^{\frac 1 2}
\sum_n  \frac{1 }{\eps+\lambda_n^{-1}} \langle 1,\varphi_n\rangle\langle h, \varphi_n\rangle.
\end{equation}
By Theorem 2.3, \cite{ChK}, for $H<\frac 1 2$, the eigenvalues satisfy the same asymptotics as in \eqref{Hge} and 
therefore form an increasing sequence, in agreement with the fact that in this case the operator $K_H$ is not compact. 
Also we have  $\langle h, \varphi_n\rangle   \sim  n^{ H-\frac 1 2}$ (Lemma 6.9 in \cite{ChK}).
By Theorem 2.3 in \cite{ChK} the averages of symmetric eigenfunctions
have asymptotics $\langle 1, \varphi_n\rangle\sim  n^{-1}$, c.f. \eqref{Hge}. 
 
Now we can estimate the growth rate of the series from \eqref{pTT}:
\begin{align*}
&
r(\eps)=\sum_{n\; \text{odd}}  \frac{1 }{\eps+\lambda_n^{-1}} \langle 1,\varphi_n\rangle\langle h, \varphi_n\rangle \sim
\sum_{n\; \text{odd}}  \frac{n^{-\frac 3 2+H} }{\eps+  n^{2H-1}}  \sim \eps^{-\frac 1 2} \int_0^\infty \frac{y^{-\frac 1 2 -H}}{y^{ 1-2H}+1}dy
\end{align*}
and hence
$$
\frac{d}{dT}\langle M\rangle_T =  p^2(T,T) \sim \text{const.}, \quad T\to\infty.
$$

Further, differentiating the series in \eqref{pTT}, we get
\begin{align*}
\frac{d}{d\eps}r(\eps)=
&
\frac{d}{d\eps}\sum_n  \frac{1 }{\eps+\lambda_n^{-1}} \langle 1,\varphi_n\rangle\langle h, \varphi_n\rangle=
\\
&
-\sum_n  \frac{n^{-\frac 3 2+H} }{(\eps+ n^{2H-1})^2} \sim
- \eps^{-\frac 3 2 }   \int_{0}^\infty\frac{y^{ \frac 1 2-3H} }{(y^{1-2H}+1)^2}d y,
\end{align*}
and
$$
\left(\frac d{dT} \log \frac{d}{dT} \langle M\rangle_T\right)^2=
\left(\frac {d}{dT}\log p(T,T)\right)^2\sim
\left(\frac {d\eps}{dT}\frac{d}{d\eps} \log \eps^{\frac 12}r(\eps)\right)^2\sim T^{-2},
$$
which proves  \eqref{ass}.
\end{proof}

\clearpage 

\subsection{Convergence of moments}\label{sec:2.2}
The convergence of moments in \eqref{LAN} 
\begin{equation}
\label{convm}
\E_\theta \left(\sqrt{T} (\widehat \theta_T-\theta)\right)^p
\xrightarrow[T\to\infty]{} \E \big(\sqrt{2|\theta|}Z\big)^p, \qquad \forall p>0,
\end{equation}
with $Z\sim N(0,1)$ holds, if $\big(\sqrt{T}(\widehat \theta_T -\theta)\big)^p$ is uniformly integrable over $T$ for all $p>0$. 
Note that  
\begin{align*}
\big(\E\big|\sqrt{T}(\widehat \theta_T -\theta)\big|^p\big)^2  &
\le
\E \left|\frac 1T \int_0^T Q_t(X)^2 d\langle M\rangle_t\right|^{-2p}
\E\left|\frac 1 {\sqrt{T}}\int_0^T Q_t(X)dM_t \right|^{2p} 
  \\
&
\le
\E \left|\frac 1 T\int_0^T Q_t(X)^2 d\langle M\rangle_t\right|^{-2p}
C_p \E\left|\frac 1 T\int_0^T Q_t(X)^2d\langle M\rangle_t \right|^{p} 
\end{align*}
where the last bound holds by the Burkholder-Davis-Gundy inequality with an absolute constant $C_p$. 
Hence \eqref{convm} holds by the de la Vall\'{e}e-Poussin theorem  if we prove that
$$
\varlimsup_{T\to \infty} \E\left|\frac 1 T\int_0^T Q_t(X)^2d\langle M\rangle_t \right|^{p} <\infty, \quad \forall p\in \mathbb{Z}.
$$
This limit is finite for $p>0$, since convergence of the Laplace transform in \eqref{enough} holds for any $\mu\in \Real$, including 
negative values. 
For $p<0$ it is finite due to bound \eqref{lem-cond:c} of Lemma \ref{lem-cond} and the 
identity 
$$
\E \left|\frac 1 T\int_0^T Q_t(X)^2d\langle M\rangle_t \right|^{- p} = \frac 1 {p!} \int_0^\infty \mu^p L_T(\mu)d\mu, \quad p\in \mathbb{N}. 
$$

\section{A concluding remark}

In the simpler regression problem
$$
X_t = \theta t + V_t,\quad t\in [0,T]
$$
the m.l.e. of $\theta\in \Real$ is given by
$$
\widehat \theta_T(X) = \frac{\int_0^T g(t,T) dX_t}{\langle M\rangle_T}.
$$
Consequently the estimation error is normal with zero mean and its variance is controlled by the growth rate of $\langle M \rangle_T$, rather than the derivative $d\langle M \rangle_T/dT$ as in the Ornstein-Uhlenbeck problem. 
Finding asymptotics of the bracket $\langle M \rangle_T$ amounts to singular perturbation analysis of the equations \eqref{WHeqmf.B.m.} and \eqref{WHeqHless} with respect to weak convergence (cf. \eqref{ueps}), which can be carried out
either directly (see the discussion concluding Section 7.1 in \cite{ChK}) or 
using the spectral asymptotics as above.

The corresponding limit variance is
$$
\E  (\widehat  \theta_T-\theta)^2 \simeq \begin{cases}
v_H T^{2H-2} & H>\frac 1 2\\
T^{-1} & H< \frac 1 2
\end{cases}\quad \text{with}\quad v_H= \frac{2H \Gamma(H+\frac 1 2)\Gamma(3-2H)}{\Gamma(\frac  3 2-H)}
$$
and it follows that the asymptotic is dominated by the fractional component for $H>\frac 1 2$ and by the standard Brownian component for $H<\frac 1 2$.

% ----------------------------------------------------------------
%\bibliographystyle{plain}
%\bibliography{/Users/Pavel/Dropbox/Pasha_Masha/bibliography/fBm.bib}

\def\cprime{$'$} \def\cprime{$'$} \def\cydot{\leavevmode\raise.4ex\hbox{.}}
  \def\cprime{$'$} \def\cprime{$'$} \def\cprime{$'$}

\end{document}